\documentclass[12pt]{amsart}
\usepackage[all]{xy}

\usepackage[margin=3.0 cm]{geometry}
\usepackage{latexsym}
\usepackage{amsmath}
\usepackage{amsfonts}
\usepackage{amssymb}
\usepackage{amsthm}
\usepackage{eucal}
\usepackage{enumerate,yfonts}
\usepackage{mathrsfs}
\usepackage{graphicx}
\usepackage{graphics}
\usepackage{epstopdf}
\usepackage{amscd}
\usepackage{bbm}
\usepackage{hyperref}
\usepackage{url}

\usepackage{bbm}
\usepackage{cancel}
\usepackage{enumerate}
\usepackage{amsmath,amsthm}
\usepackage{amssymb}
\usepackage{epsfig}
\usepackage{pstricks}
\usepackage{xy}

\DeclareSymbolFont{AMSb}{U}{msb}{m}{n}
\DeclareMathSymbol{\N}{\mathbin}{AMSb}{"4E}
\DeclareMathSymbol{\Z}{\mathbin}{AMSb}{"5A}
\DeclareMathSymbol{\R}{\mathbin}{AMSb}{"52}
\DeclareMathSymbol{\Q}{\mathbin}{AMSb}{"51}
\DeclareMathSymbol{\I}{\mathbin}{AMSb}{"49}
\DeclareMathSymbol{\C}{\mathbin}{AMSb}{"43}
\newcommand{\Cs}{\C^{x}}

\newcommand{\PP}{\mathbb{P}^1}
\newcommand{\Aone}{\mathbb{A}^1}
\newcommand{\A}{\mathcal{A}}

\newcommand{\hh}{\check{h}}

\newcommand{\Gm}{\mathbb{G}_m}
\newcommand{\g}{\mathfrak{g}}

\newcommand{\Al}{\mathcal{A}_\lambda}

\newcommand{\Am}{\mathcal{A}_\mu}

\newcommand{\T}{\mathfrak{t}^{\vee *}}
\newcommand{\Tt}{T\times\Gm}
\newcommand{\tN}{\mathfrak{t}_N^{\vee*}}

\newcommand{\GO}{G(\mathcal{O})}
\newcommand{\HO}{H(\mathcal{O})}

\newcommand{\Gr}{\mbox{Gr}_G}
\newcommand{\GrH}{\mbox{Gr}_H}

\newcommand{\GI}{G(\C[t^{-1}])}

\newcommand{\II}{I^-}
\newcommand{\Gra}{\mbox{Gra}_G}
\newcommand{\GraM}{\mbox{Gra}_M}

\newcommand{\OO}{\mathcal{O}}

\newcommand{\DO}{D_{G(\mathcal{O})}(Gr_G)}
\newcommand{\DOloop}{D_{G(\mathcal{O})\rtimes\Gm}(Gr_G)}
\newcommand{\DOH}{D_{G(\mathcal{O})\rtimes\Gm}(Gr_H)}
\newcommand{\DOZ}{D^\zeta_{G(\mathcal{O})}(Gra_G)}
\newcommand{\DOZloop}{D^\zeta_{G(\mathcal{O})\rtimes\Gm}(Gra_G)}

\newcommand{\F}{\mathcal{F}}
\newcommand{\G}{\mathcal{G}}
\newcommand{\K}{\mathcal{K}}

\newcommand{\Perv}{\mathbb{P}\mbox{erv}_N}

\newcommand{\Ext}{\mbox{Ext}^{\bullet}}
\newcommand{\Hom}{\mbox{Hom}^{\bullet}}
\newcommand{\fiber}{\mbox{Ext}^{\bullet}(\mathcal{F},\cdot)}

\newcommand{\Fiber}{\mbox{Ext}^{\bullet}_{G\times\Gm}(\mathcal{F},\cdot)}

\newcommand{\ExtFan}{\mbox{Ext}^{\bullet}(\mathcal{F},\A_n)}
\newcommand{\ExtTFA}{\mbox{Ext}^{\bullet}_{\Tt}(\mathcal{F},\Al)}

\newcommand{\ExtTFa}{\mbox{Ext}^{\bullet}_{\Tt}(\mathcal{F},\A)}
\newcommand{\ExtGFa}{\mbox{Ext}^{\bullet}_{G\times\Gm}(\mathcal{F},\A)}

\newcommand{\ExtTFan}{\mbox{Ext}^{\bullet}_{\Tt}(\mathcal{F},\A_n)}
\newcommand{\ext}{\mathcal{E}\mbox{xt}}
\newcommand{\ExtFF}{\mbox{Ext}^\bullet(\F,\F)}
\newcommand{\CohH}{H^\bullet_{\HO\rtimes\Gm}(\GrH)}
\newcommand{\CohG}{H^\bullet_{\GO\rtimes\Gm}(\Gr)}
\newcommand{\CohHH}{H^\bullet(\GrH)}
\newcommand{\extFF}{\mathcal{E}\mbox{xt}(\F,\F)}

\newcommand{\ExtTFF}{\mbox{$Ext_T$($\F$,$\F$)}}
\newcommand{\ExtTtFF}{\mbox{$Ext_{\Tt}$($\F$,$\F$)}}
\newcommand{\ExtGFF}{\mbox{Ext}^{\bullet}_{\GI\rtimes\Gm}(\F,\F)}

\newcommand{\ExtGFA}{\mbox{$Ext^{\bullet}_{G\rtimes\Gm}$($\F$,$\Al$)}}

\newcommand{\XN}{X^{*}(T^{\vee}_N)}
\newcommand{\GN}{G^{\vee}_{N}}
\newcommand{\TN}{T^{\vee}_N}
\newcommand{\MN}{M^{\vee}_{N}}
\newcommand{\gN}{\mathfrak{g}_N^\vee}

\newcommand{\Il}{\mathcal{I}_\lambda}
\newcommand{\MIl}{\mathcal{I}_{M,\lambda}}

\newcommand{\XNplus}{X^{*}_{+}(T^{\vee}_N)}

\newcommand{\HC}{\mathcal{HC}_\hbar}
\newcommand{\HCfr}{\mathcal{HC}_\hbar^{fr}}
\newcommand{\QCoh}{QCoh^{\Gm}((\T/W)^2\times\Aone)}

\newcommand{\cone}{\OO(N_{(\T/W)^2}\Delta)}

\newcommand{\CohN}{Coh^{\Gm}(N_{(\tN/W)^2}\Delta)}
\newcommand{\IC}{\mathcal{IC}}
\newcommand{\sym}{D_{perf}^{G^\vee}(\mbox{Sym}^\square(\g^\vee))}
\newcommand{\symN}{D_{perf}^{\GN}(\mbox{Sym}^\square(\gN)}

\newcommand{\alg}{\mathcal{E}}

\newtheorem{thm}{Theorem}[section]

\newtheorem{lem}[thm]{Lemma}
\newtheorem{rmk}[thm]{Remark}

\title{Twisted Satake Category}
\author{Bhairav Singh}

\begin{document}
\date{}
\maketitle

\begin{abstract}
We extend Bezrukavnikov and Finkelberg's description [BF] of the $G(\C[[t]])$-equivariant derived category on the affine Grassmannian to the twisted setting of Finkelberg and Lysenko [BF].  Our description is in terms of coherent sheaves on the twisted dual Lie algebra.  We also extend their computation of the corresponding loop rotation equivariant derived category, which is described in terms of Harish-Chandra bimodules for the twisted dual Lie algebra.  To carry this out, we have to find a substitute for the functor of global equivariant cohomology.  We describe such a functor, and show as in [BF] that it is computed in terms of Kostant-Whittaker reduction on the dual side.
\end{abstract}

\section{Introduction}

\subsection{Background and notation}

Let $G$ be a complex reductive group, $T$ a maximal torus, and $G^\vee$ its Langlands dual.  Let $\OO=\C[[t]]$, $\K=\C((t))$.  The affine Grassmanian of $G$ is an ind-scheme whose $\C$-points are $G(\K)/G(\OO)$.  One of the basic results in representation theory is the geometric Satake equivalence.  This was mostly proven by Lusztig in \cite{L1} (though not explicitly stated in as an equivalence of categories there), and then in \cite{G1} \cite{MV}

$$\mathcal{P}_{\GO}(\Gr)\simeq \mbox{Rep}(G^\vee)$$

One of the motivations for this theorem was to describe the Hecke modifications that appear in the geometric Langlands conjectures.  The present formulation of these conjectures requires a derived version of the Satake equivalence.  This is provided by \cite{BF} who prove

$$\DO\simeq \sym$$

The right hand side is the subcategory of perfect complexes in the $\G^\vee$-equivariant derived category of modules over the symmetric algebra of $\g^\vee$, considered as a differential graded algebra with zero differential (this is what $\square$ denotes).  In addition to the action of $\GO$, $\Gr$ admits and action of $\Gm$ by "loop-rotation", which in the algebraic context means scaling the parameter $t$.  The methods of Bezrukavnikov-Finkelberg also compute the derived category with this additional equivariance as

$$\DOloop\simeq D_{perf}^{G^\vee}(U^{\square}_{\hbar})$$

The right hand side is the subcategory of perfect complexes in the $\G^\vee$-equivariant derived category of modules over a graded version of the universal enveloping algebra of $\g^\vee$, considered as a differential graded algebra with zero differential.  In both cases, the introduction of a differential graded algebra structure with zero differential may seem superfluous, but it should really be thought of as a statement about $\DO$ or $\DOloop$.

The geometric Langlands correspondence is expected to admit a quantum deformation.  Corresponding to this deformation as a root of unity, an analogue of the geometric Satake equivalence was given in \cite{FL}.  To describe it we have to introduce some more notation.  Let $\Gra$ be the total space of the determinant line bundle of $\Gr$ with the zero section removed, i.e. a $\Gm$-torsor.  Fix a positive integer $N$.  Let $\hh$ be the dual Coxeter number of $G$, and $d$ the divisor of $\hh$ defined in \cite{FL}.  Let $\zeta$ be a fixed $2\hh N/d^{th}$ order character (note that this is different from $\zeta$ in \cite{FL}). Let $\Perv$ be the category such that $\Perv$[1] is the category of $\GO$-equivariant perverse sheaves on $\Gra$ with $\Gm$-monodromy $\zeta$.  Let $\XN:=\{\nu\in X_*(T)|d\iota(\nu)\in NX^*(T)\}$, and $\TN:=\mbox{Spec}\C[\XN]$.  The $\GO$-orbits $\Gra^\lambda\subset\Gra$ that support a $\zeta$-monodromic rank one local system $L\zeta$ correspond exactly to $\lambda\in X^*_+(\TN)\subset X_*^+(T)$.  Let $\Al$ denote the IC-extension of this local system from $\Gra^\lambda$.  The main result of \cite{FL} is an equivalence of tensor categories

$$\Perv\simeq \mbox{Rep}(\GN)$$

where $\GN$ is a reductive groups with maximal torus $\TN$.  The goal of this paper is to extend the results of \cite{BF} to the setting of \cite{FL}.  Specifically, we prove

\begin{thm}
There are equivalences of DG-categories (here $U_\hbar$ is the graded enveloping algebra for $\gN$)

$$\DOZ\simeq \symN$$
$$\DOZloop\simeq D_{perf}^{\GN}(U^{\square}_{\hbar})$$
\end{thm}

It will be helpful to first recall the notations and arguments of \cite{BF}, since ours closely follow theirs.  The $\GO$ orbits on $\Gr$ are indexed by dominant coweights $\lambda\in X_*^+(T)$.  Let $\IC_\lambda$ be the intersection cohomology sheaf of the closure of the corresponding orbit, and the $V_\lambda$ the corresponding representation of $G^\vee$.  Then $V_\lambda = H^\bullet(\IC_\lambda)$, and thus $V_\lambda$ has a natural action of $H^\bullet(\Gr)$ which was described in \cite{G1}.  To describe the derived category, Bezrukavnikov and Finkelberg study the action of the algebra $\CohG$ on $H^\bullet_{\GO\rtimes\Gm}(\IC_\lambda)$.  If one rewrites this as $\Ext_{\GO\times\Gm}(\C,\C)$ acting on $\Ext_{\GO\times\Gm}(\C,\IC_\lambda)$, their argument can be considered somewhat analogous to describing a (nice enough) abelian category as modules over the endomorphism algebra of a projective generator.  In this paper, all Ext's between sheaves are taken in the (equivariant) derived category.

The proof has two key steps (shown in the opposite order in \cite{BF}): to show that the natural map

$$\Ext_{\GO\rtimes\Gm}(\IC_\lambda,\IC_\mu)\rightarrow \mbox{Hom}^\bullet_{H^\bullet_{\GO\rtimes\Gm}(\Gr)} (H^\bullet_{\GO\rtimes\Gm}(\IC\lambda),H^\bullet_{\GO\rtimes\Gm}(\IC_\mu))$$

is an isomorphism and that the Ext groups are pure, and then to compute action of $H^\bullet_{\GO\times\Gm}(\Gr)$ on $H^\bullet_{\GO\times\Gm}(\IC_\lambda)$ in terms of $G^\vee$.  The first step essentially follows from \cite{G2}.  The second uses the notion Kostant-Whittaker reduction.  Let $U_\hbar$ be the graded enveloping algebra of $\g$, i.e. it has the defining relations $xy=yx=\hbar[x,y]$, and let $\HC$ be the category of integrable $U_\hbar\otimes_{\C[\hbar]}U_\hbar$-modules, and $\HCfr$ the subcategory of bimodules of the form $U_\hbar\otimes V$, $V$ a representation of $G^\vee$.  Bezrukavnikov and Finkelberg define a functor $\kappa_\hbar:\HC\rightarrow\QCoh$.  Let $\cone$ be ring of functions of the deformation to the normal cone of the diagonal in $(\T/W)^2$.  The image of an element of $\HCfr$ is always supported on $N_{(\T/W)^2}\Delta)$.  Then

$$\CohG\simeq \oplus_{Z(G^\vee)}\cone$$

and the inverse of the Satake equivalence $S$ extends to a full embedding $S_\hbar:\HCfr\rightarrow\DO$ such that

$$\kappa_\hbar\simeq H^\bullet_{\GO\rtimes\Gm}\circ S_\hbar$$

For $V\in\mbox{Rep}{G^\vee}$, let $\phi(V)=\kappa_\hbar(U_\hbar\otimes V)$ and let $V^\lambda$ denote the $\lambda$ weight space.  Let $\pi:\T\rightarrow\T/W$ be the projection, and let $\Gamma_\lambda$ be the subscheme of $\T\times\T\times\Aone$ defined by $x_2=x_1+a\lambda$.  To show $\phi(V)\simeq H^\bullet_{\GO\rtimes\Gm}(S(V))$ canonically \cite{BF} describe filtrations on $\phi(V)\otimes_{\OO(\T/W)}\OO(\T)$ and $H^\bullet_{\GO\rtimes\Gm}(S(V)) \otimes_{\OO(\T/W)}\OO(\T)\simeq H^\bullet_{T\times\Gm}(S(V))$, both with associated graded

$$\oplus_{\lambda\in X_*(T)}(Id,\pi,Id)_*\OO(\Gamma_\lambda)\otimes V^\lambda$$

They are then able to identify $\phi(V)\otimes_{\OO(\T/W)}\OO(\T)$ and $H^\bullet_{T\times\Gm}(S(V))$ inside

$$\oplus_{\lambda\in X_*(T)}(Id,\pi,Id)_*\OO(\Gamma_\lambda)\otimes V^\lambda \otimes_{\OO(\T\times\Aone)} K(\T\times\Aone)$$

using a clever reduction to rank one.

In the following we will abuse notation and use $U_\hbar$, $\HCfr$, $\kappa_\hbar$, $\phi$, etc. for the corresponding objects for $\GN$.  We hope this doesn't cause any confusion.

\subsection{Outline of the proof}

Our proof follows the above strategy, showing $\Ext_{\GO\rtimes\Gm}(\Al,\Am)$ is pure, and computing is in terms of $\GN$.  Several arise if one tries to apply the proof of \cite{BF} directly:  To show purity, one cannot directly appeal to \cite{G2} since the fixed points of $T\times\Gm$ on $\Gra$ are not isolated and since we work with twisted coefficients.  Also because of twisted coefficients (monodromic local systems) the global cohomology of an element of $\Perv$ is automatically zero.  So we need a substitute for this, and for the total cohomology of $\Gra$.  The solution is to find a substitute for the functor $\Ext(\C,\cdot)$, and hence a substitute for the constant sheaf $\C$.  The group $\GI$ acts on $\Gra$ with an open orbit $U$ through the base point, which is a $\Gm$ torsor over an infinite affine space.  This torsor is in fact trivial, and admits a rank one local system with $\Gm$-monodromy $\zeta$.  Let $\F$ be the IC-extension of this local system to $\Gra$.  This is the desired substitute for $\C_{\Gra}$ (more precisely we well take a direct sum of several similarly defined sheaves, see below).  One should really define $\F$ on the moduli space of bundles on $\PP$ or equivalently Kashiwara's "thick" Grassmannian \cite{KT}, but one can consider $\F$ as a pro-object in the category of sheaves on $\Gra$ (this seems more natural when one notes that $\C_{\Gra}$ is also a pro-object).  The open orbit $U$ intersects every $\Gra^\lambda$ in an affinely embedded open set, and one can use base change to describe $\F$ as a limit of the corresponding IC-extensions.  Then studying the action of $\ExtGFF$ on $\ExtGFA$ allows us to describe the derived category.

%\begin{thm}
%$$\DOZ\simeq D_{perf}^{\GN}(\mbox{Sym}^\square(\g_N^\vee))$$
%$$\DOZloop\simeq D_{perf}^{\GN}(U^\square_{\hbar})$$
%\end{thm}

Throughout this paper $H$ will denote the Langlands dual group of $\GN$.  One should think of the geometry of monodromic perverse sheaves on $\Gra$ as equivalent to the geometry of usual perverse sheaves on $\GrH$.  Not only is this a guiding principle, but in some cases we use an explicit comparison between the two to prove things about $\Perv$.  We will give one proof of Theorem 1.1 by directly identifying it with $\DOH$. One may also prove it by a direct comparison to Kostant reduction on $\mbox{Rep}(\GN)$ as in \cite{BF}, but even there the simplest proofs of some facts use a geometric comparison to $\GrH$.  Under the identification

$$X_*(T_H)\simeq X^*(\TN)\subset X_*(T)$$

we can naturally identify $X_*(T_H)$ with a subset of $X_*(T)$

We will denote the dominant coweight of $T_H$ corresponding to $\lambda\in X^*_+(\TN)\subset X_*^+(T)$ by $\lambda'\in X_*^+(T_H)$, and similarly for the corresponding IC-sheaves, etc.
%We will sometimes identity $\OO(\T)$, $\OO(\tN)$, and $\OO(\T_H)$ without comment.
Let $\Pi:\Gra\rightarrow\Gr$ be the projection.

\subsection{Acknowledgements}

I would like to thank my advisor Roman Bezrukavnikov for suggesting this problem, and for numerous helpful discussions.  The idea to use the sheaf $\F$ belongs to him.  I would also like to thank Ivan Mirkovic, Dennis Gaitsgory, Chris Dodd, and Tsao-Hsien Chen for useful explanations.  This paper owes a large debt to the work Bezrukavnikov and Finkelberg, from which we have borrowed several arguments.

\section{Cohomology}

This section contains the counterparts of the cohomology computations in \cite{BF}

\subsection{Cohomology}

Let $I\subset \XNplus$ be a minimal (in the dominance order) set of representatives for $Z(\GN)$.  For $i\in I$, the finite codimensional orbit $\GI\cdot i$ admits a $\zeta$-monodromic rank one local system.  Let $\F_i$ be the minimal extension of this local system to $\Gra$ (see \cite{KT} for a precise definition), and let $\F=\oplus_{i\in I}\F_i$.  As remarked in the introduction, we can consider $\F$ as a pro-object of $D_{\GI\rtimes\Gm}(\Gra)$.

A key part of argument of \cite{BF} is describing the $G$- or $G\times\Gm$-equivariant cohomology of an irreducible perverse sheaf as a module over the global equivariant cohomology ring.  In the monodromic setting, this cohomology is always zero.  It turns out a good subsitute is given by the functor $\Fiber$.  Before we compute the ring $\ExtGFF$, let us recall some properties of $\F$.

Let $W'$ be the subset of the affine Weyl group given by $\XN\rtimes W^{fin}$.  The inverse Kazhdan-Lusztig polynomials $Q_{e,w}(q)$ are equal to 1 for any $w\in W$', where $e$ is the identity element, and are zero for $w\not\in W'$.  This is well known for finite Weyl groups.  In the affine case, it follows from the pro-smoothness of the 'thick' Grassmannian (or the smoothness of $\mbox{Bun}_{\PP}$).  First consider the case when the monodromy is trivial.  Then $\F$ is just the IC-extension of the constant sheaf (on each component), so by smoothness $\F$ is also a constant sheaf.  In general, it follows from the fact that $W'$ is isomorphic to $W^{aff}$ as a Coxeter group, which is the identity on the finite part $W$ and extends the isomorphism of lattices $\Lambda\simeq\XN$.
%Let $I^{-}\subset\GI$ be the inverse image of $B^-$ under $t^{-1}\mapsto0:\GI\rightarrow G$.  The $I^{-}$-orbit of the base point is the same as the $\GI$-orbit.
By theorem 5.3.5 of \cite{KT}, above description of the Kazhdan-Lusztig polynomials says that the cohomology sheaves of the restriction of $\F$ to a $\GI$-orbit corresponding to $\XNplus$ are only nonzero in degree zero where they are of rank one, and are zero on all other orbits.

\begin{thm}
$$\ExtGFF\cong\oplus_{Z(\GN)}\OO(N_{\tN /W\times\tN /W}\Delta)$$
\end{thm}

\begin{proof}
Let $\II$ be the subgroup of $\GI$ of elements that are in $B^-$ when $t\rightarrow\infty$.  Consider the spectral sequence associated to the filtration of $\Gra$ by $\II$-orbits.  By Theorem 5.3.5. of \cite{KT}, the stalks of $\F$ satisfy parity-vanishing, so this spectral sequence degenerates.  Via $\XN\simeq X_*(T_H)$, the $\II$-orbits where $\F$ is non-zero correspond exactly to the corresponding orbits in $\GrH$, and by \cite{KT} (5.3.10), so do the degrees of the stalks/costalks.  Therefore $\ExtFF$ is computed by the same (degenerate) spectral sequence as $\CohHH$, and in particular is isomorphic to $\CohHH$ as a graded vector space.  Note that according to \cite{G1}, the latter is a polynomial algebra in variables whose degrees are twice the exponents of $\gN$.

We now argue exactly as in \cite{BF} 3.1.  We have two morphisms $pr_1^*,pr_2^*:\sum \OO(\T/W)\rightarrow \ExtGFF$, and a morphism $pr^*:\C[\hbar]\rightarrow \ExtGFF$  We claim that $pr_1^*|_{\hbar==0}pr_2^*|_{\hbar==0}$.  To see this, note by parity vanishing of $\F$, $\extFF$ is equivariantly formal, so $\ExtTFF\hookrightarrow \oplus_\lambda H_{T\times\Gm}(\lambda)$.  By \cite{BF} 3.2 (which is doesn't depend on their 3.1), the left $\OO(\T)$-action and the right $\OO(\T/W)$-action on $H_{T\times\Gm}(\lambda)$ commute when $\hbar=0$, so the morphism $(pr_1^*,pr_2^*,pr^*)$ factors through a morphism $\alpha:\oplus_{Z(\GN)}\OO(N_{\T/W\times\T/W}\Delta)\rightarrow \ExtGFF$.  Since the localization

$$\alpha_{loc}:\oplus_{Z(\GN)}\OO(N_{\T/W\times\T/W}\Delta)\otimes_{\OO(\T/W\times\Aone)}\mbox{Frac}(\OO(\T\times\Aone))$$
$$\rightarrow \ExtGFF\otimes_{\OO(\T/W\times\Aone)}\mbox{Frac}(\OO(\T\times\Aone))$$

is injective, so is $\alpha$.  Since $\ExtGFF$ has the same graded dimension as $\CohH$, which by \cite{BF} 3.1 has the same graded dimension as $\oplus_{Z(\GN)}\OO(N_{\T/W\times\T/W}\Delta)$, $\alpha$ is an isomorphism.

\end{proof}

\begin{rmk}
Notice that if $\GN$ is not adjoint, $\ExtFF$ looks like the cohomology of a disconnected space, even though $\Gra$ may be connected (e.g. for $G=SL_2$, $N=2$).  One we to see this as is follows:  by parity vanishing, the sheaf $\extFF$ is equivariantly formal in the sense of \cite{GKM1}, hence its cohomology embeds into the its cohomology restricted to the fixed points, and further $\ExtFF$ is defined as the kernel of the boundary map to $X_1$, the union of 0- and 1-dimensional orbits.  The components of $X_1$ are labeled by $Z(\GN)$.  This is why we have to take a direct sum in the definition of $\F$.
\end{rmk}

\subsection{Fiber functor}

By definition, there is a natural action of $\OO(\T\times\T/W\times\Aone)$ on $\ExtTFa$, and for $\A=\Al$ this action factors through the diagonal morphism

$$\OO(\T/W\times\T/W\times\Aone)\rightarrow \oplus_{Z(\GN)}\mathcal{O}(N_{\tN /W\times\tN /W}\Delta)\simeq \ExtTtFF$$

Let $\pi:\tN\rightarrow\tN/W$ denote the projection.  Thus $(\pi^*,Id,Id)$ is the natural map from $G$- to $T$-equivariant cohomology.  Let $\Gamma_\lambda\subset\T\times\T\times\Aone$ be the subscheme defined by the equation $x_2=x_1+a\lambda$.

Each $N_{-}(\K)$ orbit is contained in a $\GI$-orbits, so $\F$ is constant along these orbits with 1-dimensional stalks.  It is then easy to see that $\fiber$ coincides with the fiber functor of \cite{FL}.  Following \cite{BF} 3.2 we define the canonical filtration on $\ExtTFa$ for any $\A\in\Perv$.  Let $\Il$ be inverse image under $\Pi$ of the semi-infinite $N_{-}(\K)$-orbit through $\lambda$.  We filter $\ExtTFa$ by the closures of $\Il$.  The associated graded of this filtration is

$$\bigoplus_{\lambda\in\XN}\Ext_{\Il,T\times\Gm}(\F,\A)$$

Let $i_\lambda$ be the locally closed embedding of $\Il$ in $\Gra$, and $j_\lambda$ the embedding of $\Pi^{-1}\lambda$ in $\Il$.  Then

$$\Ext_{\Il,T\times\Gm}(\F,\A)=\Ext_{T\times\Gm}(j_\lambda^*i_\lambda^*\F,j_\lambda^*i_\lambda^!\A)\otimes j_\lambda^*i_\lambda^!\A$$

On $\Pi^{-1}\lambda$ both $j_\lambda^*i_\lambda^*\F$ and $j_\lambda^*i_\lambda^!\A$ reduce to $\L^\zeta$, and $\Ext_{T\times\Gm}(j_\lambda^*i_\lambda^*\F,j_\lambda^*i_\lambda^!\A)\simeq H^\bullet_{T\times\Gm}(\lambda)\simeq (Id,\pi,Id)_*\OO(\Gamma_\lambda)$ by \cite{BF} 3.2.  If $\A=S(V)$, where $V$ is a representation of $\GN$, then $j_\lambda^*i_\lambda^!\A$ is the $\lambda$ weight space of $V$, which we write $V^\lambda$.  This proves

\begin{lem}
For $V_\lambda\in\mbox{Rep}(\GN)$, the $\OO(\tN\times(\tN/W)\times\Aone)$-module $\ExtTFA$ has a canonical filtration with associated graded $\oplus_\lambda(Id,\pi,Id)_*\OO(\Gamma_\lambda)\otimes V^\lambda$, in particular $\ExtTFA$ is flat as an $\OO(\T\times\Aone)$-module.
\end{lem}

This canonical filtration is compatible with the restriction to a Levi subgroup.  Let $M$ be a Levi subgroup of $G$, and $\MN$ corresponding subgroup of $\GN$ by \cite{FL}.  Let $P^-_M$ be the parabolic subgroup of $G$ generated by $M$ and $B^-$, and let $\pi_M:\T/W_M\rightarrow\T/W$ be the projection.  Let $X_*^{+M}(T)\subset X_*(T)$ be the coweights of $T$ that are dominant for $M$.  This set indexes the $P^-_M$-orbits on $\Gra$, and the oribts which admit an $M(\OO)$-equivariant local system with monodromy $\zeta$ are indexed by

$$X^{*+}_M(\TN):=\{\lambda\in X^{+M}_*(T)| d\iota(\lambda)\in NX^*(T)\}$$

Let $\MIl$ denote the $P^-_M$-orbit through $\lambda$, and $p^\lambda:\MIl\rightarrow\GraM$ the natural projection.  Now $(\pi_M,Id,Id)^*\ExtGFa=\Ext_{L\times\Gm}(\F,\Al)$ is filtered by the $\MIl$, with associated graded

$$\bigoplus_{\lambda\in X^*_M(\TN)}\Ext_{\MIl,M\times\Gm}(\F,S(V))=$$
$$\bigoplus_{\lambda\in X^*_M(\TN)}\Ext_{M\times\Gm}(\GraM;p^\lambda_*i_\lambda^!\F,p^\lambda_*i_\lambda^!S(V))=$$
$$(Id,\pi_M,Id)_*\Ext_{M\times\Gm}(\GraM,S_M(V_{\MN}))$$

The following are proved exactly as in \cite{BF} 3.4 and 3.5.

\begin{lem}
The canonical filtration is compatible with restriction to a Levi subgroup.
\end{lem}

\begin{lem}
The canonical filtration is compatible with the tensor structure on $\Fiber$ given by convolution.
\end{lem}

\subsection{Rank one}

We begin by explaining another way to compute $\ExtTFF$ in rank one.  Let $i:X_0\rightarrow\Gra$ denotes the union of the $T$-fixed points and $j:X_1\rightarrow\Gra$ denote the union of the 0- and 1-dimensional orbits.  By a parity vanishing argument, $\extFF$ is equivariantly formal in the sense of \cite{GKM1}.  Therefore we have an exact sequence

$$0\rightarrow\ExtTFF\rightarrow\Ext_{X_0,T\times\Gm}(i^!\F,i^!\F)\rightarrow\Ext_{X_1,T\times\Gm}(j^!\F,j^!\F)$$

For a given $\lambda\in \XN=X_*(T_H)$, $i_\lambda^!\F$ is a local system of rank one.  By \cite{GKM2} there is a unique one-dimensional orbit connecting each pair of fixed points on $\Gr$.  Let $\T_{\lambda,\mu}$ be the lie algebra of the stabilizer of such a connecting orbit.  The above exact sequence says that $\ExtTFF$ is the subalgebra of functions in $(f_\lambda)\in\oplus_{\lambda}\C[\T,\hbar]\otimes \C_\lambda$ such that $f_\lambda|_{\C[\T_{\lambda,\mu}]}= f_\mu|_{\C[\T_{\lambda,\mu}]}$ for all pairs $\lambda,\mu$.  One can see that this argument applied to $\GrH$ gives the exact same answer.  Note that we can canonically identify the one-dimensional vector spaces $\C_\lambda$ and $\C_\lambda'$ by identifying generator of $\Ext_{\C^*}(L^\zeta,L^\zeta;\Z)$ and $H_{pt}(\Z)$ (here $\C^*$ means the space, not equivariance).

We now assume that $G$ has rank one, so $\GN$ does as well.

%, and $G_N^\vee\simeq SL_2$.

%In this case, if $\alpha$ is the simple root of $G$, the coroot lattice of $\GN$ is generated by $\frac{d}{N}\alpha$, so the fundamental weight of $T_H$ is $\frac{d}{N}\alpha$

\begin{lem}
$\ExtTFan\simeq H_{T_H\times\Gm}(\A_n')$ as $\OO(\tN\times\tN\times\Aone)$-modules in a way that is compatible with the canonical filtrations on each.
\end{lem}

\begin{proof}
%Since the stalks of $\F$ and $\A_n$ are 1-dimensional, we can use the localization theorem of \cite{GKM1} to canonically identify the action of $\ExtTFF$ on $\ExtTFan$ with that of $H_{T_H\times\Gm}(\GrH)$ on $H_{T_H\times\Gm}(\A_n')$, and appeal to \cite{BF} 5.1.
Using a filtration and parity the vanishing of stalks, one sees that $\ExtFan$ vanishes in odd degrees, hence the sheaf $\ext(\F,\A_n)$ is equivariantly formal.  Let $i:X_0\rightarrow\Gra$ denotes the union of the $T$-fixed points and $j:X_1\rightarrow\Gra$ denote the union of the 0- and 1-dimensional orbits.  Applying (6.3) of \cite{GKM1} gives

$$0\rightarrow\ExtTFan\rightarrow\Ext_{X_0,T\times\Gm}(i^!\F,i^!\A_n)\rightarrow\Ext_{X_1,T\times\Gm}(j^!\F,j^!\A_n)$$

Here the middle term is equal to $\oplus_{i\leq n}\OO[\T]\otimes \C_i$, and the last term by $\sum_{j<i\leq n}\OO(\T_{ij})\otimes\C_i$, where $\T_{ij}$ is the Lie algebra of the stabilizer of a 1-dimensional orbits between points over $i$ and $j$.  Here $i$ runs over the weights of the representation $V_n$ of $\GN$, and $\C_i$ is the corresponding weight space.  We can write down the corresponding exact sequence for $H$

$$0\rightarrow H_{T_H\times\Gm}(\A_n')\rightarrow H_{X_{0,H},T\times\Gm}(i^!\A_n')\rightarrow H_{X_{1,H},T\times\Gm}(j^!\A_n')$$

Here the middle sequence is also given by $\oplus_{i\leq n}\OO[\T_H]\otimes \C_i$, and the last term by $\sum_{j<i\leq n} \OO(\T_{H,ij}\otimes\C_i$, and we can canonically identify the middle and last terms of these two exact sequences, in a way that is compatible with the actions of $\ExtTFF$ and $H_{T_H\times\Gm}(\GrH)$.

This allows us to identify $\ExtTFan$ and $H_{T_H\times\Gm}(\A_n')$ as filtered $\OO(\tN\times\tN\times\Aone)$-modules without passing to the associated graded.  In particular, it implies that if $s$ is the non-trivial element of $W$, the action $s$ on the associated graded of the canonical filtration coming from $\Gra$ is the same as the action coming from $\GrH$.
\end{proof}

\begin{rmk}
The proof of the above lemma relies on the fact that the stalks of $\A_n$, $\A_n'$ are one-dimensional.  In general we know that the stalks of both $\Al$ and $\Al'$ at $\mu$ have the same dimension as the $\mu$ weight space of $V_\lambda$, but we don't have a natural way to identify the stalks themselves.  If we can define a Hopf algebra structure on $\Ext(\F,\F)$ and show that this Hopf algebra is isomorphic to the enveloping algebra of the centralizer of a principal nilpotent in $\gN$, one can use the Brylinski filtration to canonically identify the stalks of $\Al$ and $\Al'$.  It should be possible to do this using the Beilinson-Drinfeld Grassmannian.
%This will be explained in a future version.
\end{rmk}

We can now prove

\begin{thm}
We have a natural isomorphism of the $\OO(\tN/W\times/\tN/W\times\Aone)$-modules $\ExtGFA$ and $H^{\bullet}_{\HO\rtimes\Gm}(\Al')$ that preserves the canonical filtrations and induces the identity on the associated graded pieces.
\end{thm}

\begin{proof}

We follow \cite{BF} Theorem 6.  We have

$$(\pi,Id,Id)^*\ExtGFa\otimes_{\OO(\T\times\Aone)}\mathbf{k}(\T\times\Aone)=$$
$$gr(\pi,Id,Id)^*\ExtGFa\otimes_{\OO(\T\times\Aone)}\mathbf{k}(\T\times\Aone)=$$
$$\bigoplus_{\lambda}(Id,\pi,Id)_*\OO(\Gamma_\lambda)\otimes _\lambda V\otimes_{\OO(\T\times\Aone)}\mathbf{k}(\T\times\Aone)=$$
$$gr(\pi,Id,Id)^*H_{T_H\times\Gm}(\A')\otimes_{\OO(\T\times\Aone)}\mathbf{k}(\T\times\Aone)=$$
$$(\pi,Id,Id)^*H_{T_H\times\Gm}(\A')\otimes_{\OO(\T\times\Aone)}\mathbf{k}(\T\times\Aone)$$

By the proof of Lemma 4.1, and compatibility with Levi subgroups, the action of simple reflection $s\in W$ on

$$\bigoplus_{\lambda}(Id,\pi,Id)_*\OO(\Gamma_\lambda)\otimes _\lambda V\otimes_{\OO(\T\times\Aone)}\mathbf{k}(\T\times\Aone)$$

coming from

$$(\pi,Id,Id)^*\ExtGFa\otimes_{\OO(\T\times\Aone)}\mathbf{k}(\T\times\Aone)$$

is the same as the action coming from

$$(\pi,Id,Id)^*H_{T_H\times\Gm}(\A')\otimes_{\OO(\T\times\Aone)}\mathbf{k}(\T\times\Aone)$$

hence we have a naturally defined action of $W$.

Recall the submodule $\cap^{\alpha}_w w(M_\alpha)\subset \oplus_{\lambda}(Id,\pi,Id)_*\OO(\Gamma_\lambda)\otimes (_\lambda V)\otimes_{\OO(\T\times\Aone)}\mathbf{k}(\T\times\Aone)$.  By \cite{BF} 5.2 (7) $\cap^{\alpha}_w w(M_\alpha)$ contains $(\pi,Id,Id)^*H_{T_H\times\Gm}(\A')\otimes_{\OO(\T\times\Aone)}\mathbf{k}(\T\times\Aone)$, so by Lemma 4.1 it also contains $(\pi,Id,Id)^*\ExtGFa\otimes_{\OO(\T\times\Aone)}\mathbf{k}(\T\times\Aone)$.  By \cite{BF}, any section of $\cap^{\alpha}_w w(M_\alpha)$ is regular away from a codimension 2 subvariety.  By Lemma 3.1 and \cite{BF} 3.3, both $(\pi,Id,Id)^*\ExtGFa$ and $(\pi,Id,Id)^*H_{T_H\times\Gm}(\A')$ are flat $\OO(\T\times\Aone)$-modules that equal $\cap^{\alpha}_w w(M_\alpha)$ after tensoring with $\mathbf{k}(\T\times\Aone)$, hence both are in fact equal to $\cap^{\alpha}_w w(M_\alpha)$.

\end{proof}

\subsection{Comparison of the Kostant functor with cohomology}

Recall the Kostant functor $\kappa_\hbar:\HC\rightarrow\CohN$ for $\GN$ from \cite{BF}.  We can now prove the following

\begin{thm}
Let $S_\hbar:\mbox{Rep}(\GN)\rightarrow \mathcal{P}_{\GO\rtimes\Gm}(\Gra)$ be the inverse of the twisted Satake equivalence.  $S_\hbar$ extends to a full embedding $\S_\hbar:\HC^{fr}\rightarrow \DOZ$ such that

$$\kappa_\hbar\simeq H^\bullet_{\GO\rtimes\Gm}\circ S_\hbar$$
\end{thm}

\begin{proof}
By theorem 6.1 below $H^\bullet_{\GO\rtimes\Gm}$ is fully faithful on the image of $S_\hbar$, and $\kappa_\hbar$ is exact and fully faithful by Lemma 4 of \cite{BF}.  By Lemma 4.1 and Theorem 2 of \cite{BF}, we know the result in rank 1.  We can then conclude exactly as in the proof of \cite{BF} Theorem 6 (or our Theorem 2.8).
\end{proof}

\section{Purity and the derived category}

Using a $\Gm$-action, one may show that the $\Al$ are pointwise pure (see e.g. \cite{KT}, remark preceding Theorem 5.3.5).  Recall the following "adjunction" formula from \cite{FL}

$$\Ext(\A_1\ast\A_2,\A_3)\simeq\Ext(\A_1,\A_3\ast\A_2^\vee)$$

Taking $\A_1=\A_0$, $\A_2=\Al$, and $\A_3=\Am$ we have

$$\Ext(\Al,\Am)\simeq\Ext(\A_0\ast\Al,\Am)\simeq\Ext(\A_0,\Am\ast\Al^\vee)\simeq\bigoplus_{\nu}\Ext(\A_0,\A_\nu)$$

The pointwise purity of the $\A_\nu$ at $\tilde{0}$ then implies that $\Ext(\Al,\Am)$ is pure.  To extend this to $T\times\Gm$-equivariant Ext's, let $P_i$ be finite dimensional approximations to the classifying space of $T\times\Gm$, so that we have ind-varieties $P_i\Gra$ fibred over $P_i$ with fiber $\Gra$ such that $P_i\A$ be the sheaf on $P_i\Gra$ giving the equivariant structure on $\A$.  By defining a suitable $\Gm$-action, one can show the sheaf $P_i\A$ is also pointwise pure.  By definition,

$$\Ext_{T\times\Gm}(\Al,\Am)=\lim_{\rightarrow}\Ext_{P_i\Gra}(P_i\Al,P_i\Am)$$

We have a corresponding 'adjunction' formula on each $P_i\Gra$, due to the degeneracy of the Leray-Hirsch spectral sequence, which allows us to prove purity by the same argument in the equivariant case.

We now turn to the computation of $\Ext_{T\times\Gm}(\Al,\Am)$.  We have a canonical map

$$\alpha:\Ext_{T\times\Gm}(\Al,\Am)\rightarrow\Hom_{\ExtTtFF}(\Ext_{T\times\Gm}(\F,\Al),\Ext_{T\times\Gm}(\F,\Am))$$

\begin{thm}
$\alpha$ is an isomorphism of graded vector spaces.
\end{thm}

\begin{proof}

Let $\beta$ be $\alpha$ composed with the natural inclusion

$$\Hom_{\ExtTFF}(\Ext_{T\times\Gm}(\F,\Al),\Ext_{T\times\Gm}(\F,\Am))\hookrightarrow \Hom_{\C}(\Ext_{T\times\Gm}(\F,\Al),\Ext_{T\times\Gm}(\F,\Am))$$

Consider the filtration given by $I$-orbits indexed by $p$, along which $\Al$ is locally constant.  By parity considerations this spectral sequence degenerates.  Suppose we have $f\in\Ext_{T\rtimes\Gm}(\Al,\Am)$ such that $\beta(f)=0$.  Since the spectral sequence degenerates, for each $p$ we have

$$0=\beta(i_p^!f)\in \Hom_{\C}(\Ext_{T\rtimes\Gm}(\F,i_p^!\Al),\Ext_{T\rtimes\Gm}(\F,i_p^!\Am))$$

Since each orbit $p$ contains at most a single fixed point, this implies that $0=i_p^!f\in\Ext_{T\rtimes\Gm}(i_{p!}i_p^!\Al,i_{p!}i_p^!\Am)$.  By induction on the filtration as in \cite{BGS} 3.4.2, we conclude that $f=0$, so $\beta$ and hence $\alpha$ is injective.  To see that $\alpha$ is surjective, it suffices to note that both sides have the same dimension.  By Theorem 2.8

$$\Ext_{T\times\Gm}(\Al,\Am)\simeq \Ext_{T_H\times\Gm}(\Al',\Am')$$
$$\Hom_{\ExtTFF}(\Ext_{T\times\Gm}(\F,\Al),\Ext_{T\times\Gm}(\F,\Am))\simeq \Hom_{H_{T_H\times\Gm}}(H_{T_H\times\Gm}(\Al'),H_{T_H\times\Gm}(\Am'))$$

and by \cite{BF} we have

$$\Ext_{T_H\times\Gm}(\Al',\Am')\simeq \Hom_{H_{T_H\times\Gm}(\GrH)}(H_{T_H\times\Gm}(\Al'),H_{T_H\times\Gm}(\Am'))$$

Putting these together gives the desired equality of dimensions.

We can now prove Theorem 1.1.
%By Theorem 2.8
%$$\Hom_{\ExtTFF}(\Ext_{T\times\Gm}(\F,\Al),\Ext_{T\times\Gm}(\F,\Am))\simeq %\Hom_{H_{T_H\times\Gm}(\GrH)}(H_{T_H\times\Gm}(\Al'),H_{T_H\times\Gm}(\Am'))$$
Let $\alg_\lambda$ be the endomorphism algebra of the sum of the simple objects of $\DOZ$ over $\mu\leq\lambda$.  Let $\alg_\lambda'$ be the corresponding algebra for $\DOH$.  The above that $\alg_\lambda\simeq\alg_\lambda'$.  Let $\DOZ^{\leq\lambda}$ be the sub triangulated category of objects supported on orbits $\leq\lambda$.  By the purity of the equivariant Ext groups, and Proposition 6 of \cite{BF}, we have a natural functor $\Phi_\lambda:D_{perf}(\alg_\lambda) \rightarrow\DOZ^{\leq\lambda}$ that sends the free module to the $\oplus_{\mu\leq\lambda}\Am$, and induces and isomorphism on Ext groups by Theorem 6.1, so $\Phi_\lambda$ is an isomorphism.  An identical argument for $\alg_\lambda'$, and taking the direct limit over $\lambda$ gives

$$\DOZloop\simeq \lim_{\lambda}D_{perf}(\alg_\lambda)\simeq \lim_{\lambda}D_{perf}(\alg'_\lambda)\simeq \DOH\simeq D_{perf}^{\GN}(U^{\square}_{\hbar})$$

Similar arguments give the result for $\DOZ$.

\begin{rmk}
Given Theorem 2.9 and Theorem 3.1 we can also prove Theorem 1.1 using the argument of \cite{BF} 6.6.
\end{rmk}

\end{proof}

\end{document}